\documentclass[12pt]{amsart}

\newtheorem{theorem}{Theorem}[section]
\newtheorem{lemma}[theorem]{Lemma}
\newtheorem{proposition}[theorem]{Proposition}
\newtheorem{corollary}[theorem]{Corollary}   
\newtheorem{definition}[theorem]{Definition}
\newtheorem{example}[theorem]{Example}

\newtheorem{remark}[theorem]{Remark}

\numberwithin{equation}{section}

\usepackage[utf8]{inputenc}
\usepackage[english]{babel}
\usepackage{times}
\usepackage{enumerate}
\usepackage{tfrupee}
\usepackage{tikz}
\usetikzlibrary{chains,fit}
\usetikzlibrary{shapes,snakes}
\usetikzlibrary{graphs}
\usepackage{graphicx,adjustbox}
\usepackage{hyperref}
\usepackage{amsmath,amssymb}
\usepackage{amscd}
\usepackage[all,cmtip]{xy}
\usepackage{calrsfs} 
\usepackage{tabularx}
\usepackage{amsthm}
\usepackage{mathalfa,mathrsfs}
\usepackage[all]{xy}
\begin{document}
\title{Binomial edge ideals of Clutters}
\author{
Kamalesh Saha
\and
Indranath Sengupta
}
\date{}

\address{\small \rm  Discipline of Mathematics, IIT Gandhinagar, Palaj, Gandhinagar, 
Gujarat 382355, INDIA.}
\email{kamalesh.saha@iitgn.ac.in}

\address{\small \rm  Discipline of Mathematics, IIT Gandhinagar, Palaj, Gandhinagar, 
Gujarat 382355, INDIA.}
\email{indranathsg@iitgn.ac.in}
\thanks{The third author is the corresponding author; supported by the 
MATRICS research grant MTR/2018/000420, sponsored by the SERB, Government of India.}

\date{}



\allowdisplaybreaks

\begin{abstract}
In this paper, we introduce the notion of binomial edge ideals of 
a clutter and obtain results similar to those obtained 
for graphs by Rauf \& Rinaldo in \cite{raufrin}. We also 
answer a question posed in their paper.
\end{abstract}

\maketitle

\section{Introduction}
The notion of edge ideals of simple graphs 
was introduced by Villarreal in \cite{vi} and the Cohen-Macaulay property 
was studied in great detail. Subsequently, many authors have worked on similar 
problems, for example \cite{hh}, \cite{sivavi}. Later the edge 
ideal was generalized for clutters and it was observed that the 
square free monomial ideals are in one to one corresponds with 
set of clutters (see \cite{vil}). Binomial edge ideals of graphs 
was introduced independently in 
\cite{hhhrkara} and \cite{oh} in 2010. In \cite{hhhrkara}, \cite{crri} 
and \cite{enhh}, some connections were established between the algebraic 
properties of binomial edge ideals and combinatorial properties of the 
underlying graph. In \cite{enhh}, the Cohen-Macaulay property of some 
special graphs were studied. Full classification of Cohen-Macaulay 
binomial edge ideals in terms of the underlying graph is still a wide 
problem. 
\medskip

Some constructions of Cohen-Macaulay binomial edge ideals using gluing 
of graphs and cone on graphs have been done in \cite{raufrin}. Our aim 
in this article is to show that a similar construction can be carried 
out for clutters after defining the equivalent notions for clutters. 
In this article, we introduce the notion of binomial edge ideals of 
a clutter and obtain results similar to those obtained for graphs 
in \cite{raufrin} and finally answer a question 
posed in  \cite{raufrin}. The paper is arranged in the following fashion: 
We first introduce a few basic concepts in the 
next section, e.g. path in a clutter, connected 
clutter, cut point of a clutter and finally define binomial edge ideals for clutters. In this section we also established some properties regarding the binomial edge ideals of clutters and introduce some new notations. 
In section 3, we introduce the notion of gluing for clutters and 
establish similar results as in \cite{raufrin}. We show that if $\mathcal{C}$ 
is the gluing of the clutters $\mathcal{C}_{1}$ and $\mathcal{C}_{2}$, then $J_{\mathcal{C}}$ is 
Cohen-Macaulay if and only if $J_{\mathcal{C}_{1}}$ and $J_{\mathcal{C}_{2}}$ 
are both Cohen-Macaulay. We finally generalize the gluing concept for 
more than two clutters. In section 4, we introduce the notion of cone 
on a clutter and prove similar results as in \cite{raufrin} for cones 
on graphs. There is an open question in \cite{raufrin}, whether the 
converse of the Theorem 3.8 (\cite{raufrin}) is true or not. In this paper, 
we prove that the converse is true for graphs as well as clutters 
in Theorem \ref{coneconverseraufrin}. After writing the paper, we 
came to know that recently the same question has also been answered in 
\cite{bms} for graphs. However, we have proved it for a more general concept 
of clutters and our main motivation has been to generalise the results of 
\cite{raufrin} in the context of clutters.

\section{Preliminaries}
Let $S = K[x_{1},\ldots,x_{n},y_{1},\ldots,y_{n}]$ be the polynomial ring in $2n$ variables with coefficients in a field $K$. Let $G$ be a graph on the vertex set $[n] = \{1, 2, \ldots, n\}$. 
For each edge $\{i,j\}$ of $G$, with $i < j$, we associate the binomial 
$f_{ij}=x_{i}y_{j}-x_{j}y_{i}$. The ideal $J_{G}$ of $S$ generated by 
$f_{ij}$ such that $i < j$ , is called the \textit{binomial edge ideal} of $G$. A binomial edge ideal of a graph has a natural determinantal structure in the sense that it can be seen as an ideal generated by a set of $2\times 2$-minors of a $2 \times n$ matrix $X$ of indeterminates. For example, the ideal generated by 
all $2\times 2$-minors of $X$ is the binomial edge ideal of the complete 
graph on $[n]$. 

\begin{definition}\label{clutter} {\rm 
A \textit{clutter} $\mathcal{C}$ on a finite set $V(\mathcal{C})$, 
called the vertex set, is a collection $E(\mathcal{C})$, called 
the edge set, of subsets of $V(\mathcal{C})$ with the property 
that $e_{1}\not\subseteq e_{2}$ for all $e_{1},e_{2}\in E(\mathcal{C})$.
}
\end{definition}

We now introduce the notions of a binomial edge ideal, a path, 
a clique and the associated graph for a clutter. These definitions 
are inspired by similar concepts associated to graphs and have 
been used extensively in our work generalising the results proved 
in the paper \cite{raufrin}.

\begin{definition}\label{beic}{\rm 
Let $S = K[x_{1},\ldots,x_{n},y_{1},\ldots,y_{n}]$ be
the polynomial ring in $2n$ variables with coefficients in a field $K$. Let $\mathcal{C}$ be a clutter on the vertex set $[n]$. For each edge 
$e\in E(\mathcal{C})$ and each pair $i,j\in e$ with $i<j$, we associate a binomial $f_{ij}=x_{i}y_{j}-x_{j}y_{i}$. Consider the ideal $J_{\mathcal{C}}$ of $S$ generated by $f_{ij}=x_{i}y_{j}-x_{j}y_{i}$ such that $i<j$. We will call this ideal $J_{\mathcal{C}}$, \textit{the binomial edge ideal} of $\mathcal{C}$.
}
\end{definition}

\begin{definition}\label{cliq}{\rm 
Let $\mathcal{C}$ be a clutter with vertex set $V(\mathcal{C})$ and 
edge set $E(\mathcal{C})$. A subset $D\subset V(\mathcal{C})$ is 
called a \textit{clique} of $\mathcal{\mathcal{C}}$ if for all 
$i$ and $j$ belonging to $D$, with $i\not =j$, there exists an edge 
$e\in E(\mathcal{C})$ containing $i$ and $j$. Note that every edge 
of a clutter is a clique and also singletons are cliques. If 
$V(\mathcal{C})$ itself is a clique, then we call $\mathcal{C}$ a 
complete clutter.
}
\end{definition}

\begin{example}\label{cliqex}{\rm 
Let $\mathcal{C}$ be a clutter with the vertex set $V(\mathcal{C})=\{1,2,3,4,5,6\}$ and the edge set $E(\mathcal{C})=\{\{1,2,4\}, \{2,4,6\}, \{4,5\}, \{1,3,6\}\}$. Then $D=\{1,2,4,6\}$ is a clique and it is also maximal.
}
\end{example}

\begin{definition}\label{path}{\rm 
A \textit{path} in a clutter is a sequence of vertices $\{i_{1}, i_{2}, \ldots , i_{m}\}$, with the property that each vertex $i_{j}$ is adjacent to the vertex 
$i_{j+1}$. Adjacency of two vertices means that there exists an edge containing 
those vertices. A clutter is said to be \textit{connected} if for every pair of 
vertices $l\neq k$, there exists a path $\{i_{1}, i_{2}, \ldots , i_{k}\}$, such that $i_{1} = l$ and $i_{m}=k$.
}
\end{definition}

\begin{definition}\label{assograph}{\rm 
For a clutter $\mathcal{C}$, let $G^{\mathcal{C}}$ be the graph with the 
vertex set $V(G^{\mathcal{C}}) = V(\mathcal{C})$ and the edge set $E(G^{\mathcal{C}})= \{\{i,j\} \mid \exists \, e\in E(\mathcal{C}); \{i,j\}\subseteq e\}$. The graph $G^{\mathcal{C}}$ is called 
\textit{the associated graph} of $\mathcal{C}$.
}
\end{definition}

\begin{remark}\label{re1}
Note that $J_{\mathcal{C}}=J_{G^{\mathcal{C}}}$.
\end{remark} 

\begin{proposition}\label{prop1}
A clutter $\mathcal{C}$ is connected if and only if the associated graph $G^{\mathcal{C}}$ is connected.
\end{proposition}

\begin{proof}
Let $\mathcal{C}$ be a connected clutter and $i,j\in V(G^{\mathcal{C}})=V(\mathcal{C})$. There 
exists a path in the clutter between $i,j$, since 
$\mathcal{C}$ is connected, and this is a path between 
$i,j$ in $G^{\mathcal{C}}$ by the construction of 
$G^{\mathcal{C}}$. Hence $G^{\mathcal{C}}$ is connected graph.
\medskip

Conversely, let $G^{\mathcal{C}}$ be connected and $x_{1},x_{n}\in V(\mathcal{C})=V(G^{\mathcal{C}})$. Since $G^{\mathcal{C}}$ is 
connected, there exists a path between $x_{1},x_{n}$ in 
$G^{\mathcal{C}}$, say $x_{1},x_{2},\ldots,x_{n}$. Now 
$\{ x_{i},x_{i+1}\}\in E(G^{\mathcal{C}})$ for all 
$i\in\{ 1,2\ldots n-1\}$. By the definition of 
$G^{\mathcal{C}}$, there exists an edge in 
$E(\mathcal{C})$ containing $x_{i},x_{i+1}$ for 
each $i\in\{ 1,2\ldots n-1\}$. Therefore, $x_{i},x_{i+1}$ 
are adjacent for all $i\in\{ 1,2\ldots n-1\}$, i.e., 
$x_{1},x_{2},\ldots,x_{n}$ is a path in $\mathcal{C}$ also. 
Hence, $\mathcal{C}$ is connected.
\end{proof} 

\begin{definition}\label{newclutter}{\rm  
Let $\mathcal{C}$ be a clutter and $v\in V(\mathcal{C})$. Consider 
a new clutter, denoted by $\mathcal{C}/v$, such that 
$V(\mathcal{C}/v)= V(\mathcal{C})\setminus\{ v\}$ and 
$E(\mathcal{C}/v)$ is defined as the set containing the maximal elements 
of the set $\{ e\setminus\{ v\} \mid e\in E(\mathcal{C})\}$. 
It follows from the definition that 
$G^{\mathcal{C}/v}=G^{\mathcal{C}}\setminus\{v\}$.
}
\end{definition}

\begin{definition}\label{cutp}{\rm 
Let $\mathcal{C}$ be a clutter and $v\in V(\mathcal{C})$. Then $v$ is said 
to be a \textit{cut point} of $\mathcal{C}$ if number of connected components in $G^{\mathcal{C}/v}$ is more than the number of connected components of 
$G^{\mathcal{C}}$.
}
\end{definition}

\begin{proposition}\label{prop2}
$v$ is a cut point of a clutter $\mathcal{C}$ if and only if $v$ is a 
cut point of $G^{\mathcal{C}}$.
\end{proposition}

\begin{proof}
It is enough to prove this result for connected clutters. 
Let $\mathcal{C}$ be a connected clutter. Then, $v$ is a 
cut point of $\mathcal{C}$ if and only if $\mathcal{C}/v$ 
is not connected, if and only if $G^{\mathcal{C}/v}=G^{\mathcal{C}}\setminus\{v\}$ is not connected (by Proposition \ref{prop1}), if and only if 
$v$ is a cut point of $G^{\mathcal{C}}$.
\end{proof}

Let $\mathcal{C}_{1}$ and $\mathcal{C}_{2}$ be two clutters. We set $\mathcal{C}=\mathcal{C}_{1}\cup \mathcal{C}_{2}$ (resp.  $\mathcal{C}=\mathcal{C}_{1}\sqcup \mathcal{C}_{2}$; the disjoint union of $\mathcal{C}_{1}$ and $\mathcal{C}_{2}$) where $\mathcal{C}$ is the clutter with $V(\mathcal{C}) = V(\mathcal{C}_{1})\cup V(\mathcal{C}_{2})$ (resp. $V(\mathcal{C}) = V(\mathcal{C}_{1})\sqcup V(\mathcal{C}_{2})$) and $E(\mathcal{C})$ is the set containing the maximal elements 
of $E(\mathcal{C}_{1})\cup E(\mathcal{C}_{2})$ (resp. $E(\mathcal{C}) = E(\mathcal{C}_{1})\sqcup E(\mathcal{C}_{2})$).

\begin{proposition}\label{prop3}
Let $\mathcal{C},\mathcal{C}_{1},\mathcal{C}_{2}$ be clutters. If 
$\mathcal{C}=\mathcal{C}_{1}\cup \mathcal{C}_{2}$ then $G^{\mathcal{C}}=G^{\mathcal{C}_{1}}\cup G^{\mathcal{C}_{2}}$.
\end{proposition}
\begin{proof}
Clear from the definition of the associated graph of a clutter.
\end{proof}

\begin{definition}\label{simlicialcomp}{\rm  
A \textit{simplicial complex} on a vertex set 
$V=\{ x_{1},\ldots,x_{n}\}$ is a collection of 
subsets of $V$, such that the following properties hold:

\begin{enumerate}
\item[(i)] $\{ x_{i}\}\in \Delta$ for all $x_{i}\in\Delta$;
\item[(ii)] $F\in\Delta$ and $G\subseteq F$ imply $G\in\Delta$.
\end{enumerate}

\noindent An element $F\in \Delta$ is called a \textit{face} of $\Delta$. A maximal 
face of $\Delta$ with respect to inclusion is called a \textit{facet} of 
$\Delta$. A vertex $i$ of $\Delta$ is called a \textit{free vertex} of $\Delta$ if 
$i$ belongs to exactly one facet.
}
\end{definition}

\begin{definition}\label{cliquecomplex}{\rm  
The \textit{clique complex} $\Delta(\mathcal{C})$ of a clutter $\mathcal{C}$ 
is the simplicial complex whose faces are the cliques of $\mathcal{C}$. 
Hence, a vertex $v$ of a clutter $\mathcal{C}$ is called \textit{free vertex} 
if it belongs to only one maximal clique of $\Delta(\mathcal{C})$.
}
\end{definition}

We borrow a few notations and results from \cite{raufrin}, which would be 
required for our purpose. Let $T\subseteq [n]$, and 
$\overline{T}=[n]\setminus T$. We define the induced clutter on 
$\overline{T}$, denoted by $\mathcal{C}_{\overline{T}}$, such that 
$V(\mathcal{C}_{\overline{T}})=\overline{T}$ and $E(\mathcal{C}_{\overline{T}})$ 
is the set containing all the maximal elements of 
$\{ e\setminus T \mid e\in E(\mathcal{C})\}$.

\begin{proposition}\label{prop4}
Let $\mathcal{C}$ be a clutter and $T\subseteq V(\mathcal{C})$. Then $G^{\mathcal{C}_{T}}=G^{\mathcal{C}}_{T}$.
\end{proposition}

\begin{proof}
We have $V(G^{\mathcal{C}_{T}})=V(\mathcal{C}_{T})=T$ and 
$V(G^{\mathcal{C}}_{T})=T$. Therefore  
$V(G^{\mathcal{C}_{T}})=V(G^{\mathcal{C}}_{T})$. 
To show that $E(G^{\mathcal{C}_{T}})=E(G^{\mathcal{C}}_{T})$, 
let $\{ i,j\}\in E(G^{\mathcal{C}_{T}})$, then there 
exists $e\in E(\mathcal{C}_{T})$, such that $i,j\in e$. Therefore, 
there exists $e^{\prime}\in E(\mathcal{C})$ such that 
$e=e^{\prime}\setminus \overline{T}$. This implies that 
$i,j\in e^{\prime}$ and so $\{ i,j\}\in E(G^{\mathcal{C}})$. 
Since $\{ i,j\}\subset T$, we have 
$\{ i,j\}\in E(G^{\mathcal{C}}_{T})$. 
Thus $E(G^{\mathcal{C}_{T}})\subseteq E(G^{\mathcal{C}}_{T})$. 
Now, let $\{ i,j\}\in E(G^{\mathcal{C}}_{T})$. 
Then $\{ i,j\}\in E(G^{\mathcal{C}})$, i.e., 
there exists $e\in E(\mathcal{C})$ containing $i,j$. We get 
$e^{\prime}=e\setminus \overline{T}\in E(\mathcal{C}_{T})$ 
and $i,j\in e^{\prime}$ as $i,j\in T$. This gives 
$\{ i,j\}\in E(G^{\mathcal{C}_{T}})$. Hence, 
$E(G^{\mathcal{C}_{T}})=E(G^{\mathcal{C}}_{T})$.
\end{proof}

We write $P_{T}(\mathcal{C})=P_{T}(G^{\mathcal{C}})$. Then from 
\cite{raufrin} (section 1) we have 
$$J_{\mathcal{C}}=J_{G_{\mathcal{C}}}=\cap_{T\subset [n]}P_{T}(G^{\mathcal{C}})=\cap_{T\subset [n]}P_{T}(\mathcal{C}).$$
If $T$ has cut point property for $G^{\mathcal{C}}$, then we say $T$ has cut point property for $\mathcal{C}$. We denote by $\mathscr{M}(\mathcal{C})$ the set of minimal prime ideals of $J_{\mathcal{C}}$ and by $\mathscr{C}(\mathcal{C})$ the set of all $T\subset V(\mathcal{C})$ such that $T$ has cut point property for $\mathcal{C}$.

\begin{lemma}\label{lem1}
$P_{T}(\mathcal{C})\in \mathscr{M}(\mathcal{C})$ if and only if $T\in\mathscr{C}(\mathcal{C})$.
\end{lemma}

\begin{proof}
Since $P_{T}(\mathcal{C})=P_{T}(G^{\mathcal{C}}),\, \mathscr{M}(\mathcal{C})=\mathscr{M}(G^{\mathcal{C}})$ and $\mathscr{C}(\mathcal{C})=\mathscr{C}(G^{\mathcal{C}})$, the proof follows from Corollary 3.9 (\cite{hhhrkara}).
\end{proof}

\section{Gluing of Clutters and Binomial Edge Ideals}
The unmixed property and the Cohen-Macaulay property of an 
edge ideal of a graph, constructed by gluing of two graphs, 
was studied in \cite{raufrin}. In this section, we study 
the same two properties of a binomial edge ideal of a clutter, 
constructed by gluing of two clutters with respect to a free 
vertex belonging to both the clutters. It is known that a 
binomial edge ideal $J_{\mathcal{C}}$ is Cohen-Macaulay 
(resp. unmixed) if and only if $J_{H}$ is Cohen-Macaulay 
(resp. unmixed), for each connected component $H$ of 
$\mathcal{C}$; this allows us to assume that the clutter 
$\mathcal{C}$ is connected. 

\begin{proposition}\label{glu1}
Let $\mathcal{C}$ be a clutter. Then $\Delta(\mathcal{C})=\Delta(G^{\mathcal{C}})$.
\end{proposition}

\begin{proof}
Let $F\in \Delta(\mathcal{C})$ be a face. Let $i,j\in F$ and $i\neq j$. There 
exists $e\in E(\mathcal{C})$, such that $i,j\in e$. Then 
$\{ i,j\} \in E(G^{\mathcal{C}})$. Therefore, for all 
$i,j\in F$ with $i\neq j$, we have $\{ i,j\}\in E(G^{\mathcal{C}})$, which implies that $F$ is a clique of $G^{\mathcal{C}}$. It follows 
that $F\in \Delta(G^{\mathcal{C}})$ and so $\Delta(\mathcal{C})\subseteq\Delta(G^{\mathcal{C}})$.
\medskip

Now let $F\in \Delta(G^{\mathcal{C}})$. Then $F$ is a clique of 
$G^{\mathcal{C}}$, i.e., for all $i,j\in F$ with $i\neq j$, 
$\{ i,j\}\in E(G^{\mathcal{C}})$. By the definition 
of $G^{\mathcal{C}}$, there exists $e\in E(\mathcal{C})$ containing 
$i,j$, for all $i,j\in F$ and $i\not=j$. It follows that $F$ is a 
clique of $\mathcal{C}$ and so $F\in\Delta(\mathcal{C})$. Hence 
$\Delta(\mathcal{C})=\Delta(G^{\mathcal{C}})$.
\end{proof}

\begin{corollary}\label{glu2}
$v$ is a free vertex of $\Delta(\mathcal{C})$ if and only if $v$ is a free vertex of $\Delta(G^{\mathcal{C}})$.
\end{corollary}

\proof Follows from the definitions. \qed

\begin{proposition}\label{glu3}
Let $\mathcal{C}$ be a clutter, $\Delta(\mathcal{C})$ its clique complex and $v\in V(\mathcal{C})$. The following statements are equivalent:
\begin{enumerate}
\item[(a)] There exists $T\in \mathscr{C}(\mathcal{C})$, such that $v\in T$;
\item[(b)] $v$ is not a free vertex of $\Delta(\mathcal{C})$.
\end{enumerate}
\end{proposition}

\begin{proof}
Since  $\mathscr{C}(\mathcal{C})=\mathscr{C}(G^{\mathcal{C}})$ and 
$\Delta(\mathcal{C})=\Delta(G^{\mathcal{C}})$, the proof follows from \cite{raufrin} (Proposition 2.1).
\end{proof}

\begin{lemma}\label{glu4}
Let $\mathcal{C}$ be a clutter with $v\in V(\mathcal{C})$, such that 
$v$ is a free vertex in $\Delta(\mathcal{C})$. Let $F$ be the facet 
of $\Delta(\mathcal{C})$, with $v\in F$, and let $T\subset V(\mathcal{C})$ 
with $F\setminus\{ v\}\not\subseteq T$. The following conditions are equivalent:
\begin{enumerate}
\item[(a)] $T\in\mathscr{C}(\mathcal{C})$;
\item[(b)] $v\not\in T$ and $T\in \mathscr{C}(\mathcal{C}/v)$.
\end{enumerate}
\end{lemma}

\begin{proof}
We know that $\Delta(\mathcal{C})=\Delta(G^{\mathcal{C}})$. Therefore 
$v$ is a free vertex of $\Delta(\mathcal{C})$ if and only if $v$ 
is a free vertex of $\Delta(G^{\mathcal{C}})$. 
\medskip

\noindent\textsf{(a) $\Longrightarrow$ (b):} By an application 
of \cite{raufrin} (Lemma 2.2), we get 
$T\in \mathscr{C}(\mathcal{C})=\mathscr{C}(G^{\mathcal{C}})$, 
which implies $v\not\in T$ and 
$T\in \mathscr{C}(G^{\mathcal{C}}\setminus \{ v\})$. 
Now the result follows from the fact that 
$G^{\mathcal{C}/v}=G^{\mathcal{C}}\setminus \{ v\}$.
\medskip

\noindent\textsf{(b) $\Longrightarrow$ (a):} $v\not\in T$ and $T\in \mathscr{C}(\mathcal{C}/v)=\mathscr{C}(G^{\mathcal{C}/v})=\mathscr{C} (G^{\mathcal{C}}\setminus \{ v\})$. Then $T\in\mathscr{C}(G^{\mathcal{C}})=\mathscr{C}(\mathcal{C})$, by \cite{raufrin} (Lemma 2.2).
\end{proof} 

\begin{definition}\label{gluing}{\rm 
Let $\mathcal{C}=\mathcal{C}_{1}\cup \mathcal{C}_{2}$ be a clutter such that $V(\mathcal{C}_{1})\cap V(\mathcal{C}_{2})=\{ v\}$ and v is a free 
vertex of $\Delta(\mathcal{C}_{1})$ and $\Delta(\mathcal{C}_{2})$. We 
say that $\mathcal{C}$ is a gluing of $\mathcal{C}_{1}$ and $\mathcal{C}_{2}$.
}
\end{definition}

\begin{lemma}\label{glu5}
Let $\mathcal{C}$ be a gluing of the clutters $\mathcal{C}_{1}$ and 
$\mathcal{C}_{2}$ at the vertex $v$. Let $v\in F_{1}\in\Delta(\mathcal{C}_{1})$ 
and $v\in F_{2}\in\Delta(\mathcal{C}_{2})$, where $F_{1},\, F_{2}$ are facets. 
Then 
$$ \mathscr{C}(\mathcal{C})=\mathscr{A}\cup\mathscr{B},$$ where 
$$\mathscr{A}=\{ T\subset V(\mathcal{C}) \mid T=T_{1}\cup T_{2},\,T_{i}\in\mathscr{C}(\mathcal{C}_{i})\,\,for\,\, i=1,2\},$$ and $$\mathscr{B}=\{ T\subset V(\mathcal{C}) \mid T=T_{1}\cup T_{2}\cup \{ v\},\, T_{i}\in\mathscr{C}(\mathcal{C}_{i}), \, F_{i}\not\subseteq T_{i}\cup\{ v\},\, i=1,2\}.$$
\end{lemma}

\begin{proof}
$G^{\mathcal{C}}$ is a gluing of $G^{\mathcal{C}_{1}}$ and $G^{\mathcal{C}_{2}}$ at $v$. Since $\Delta(\mathcal{C}_{i})=\Delta(G^{\mathcal{C}_{i}})$, $\mathscr{C}(\mathcal{C})=\mathscr{C}(G^{\mathcal{C}})$ and $\mathscr{C}(\mathcal{C}_{i})=\mathscr{C}(G^{\mathcal{C}_{i}})$, for $i=1,2$, the proof follows from \cite{raufrin} (Lemma 2.3).
\end{proof}

\begin{corollary}\label{glu6}
Let $\mathcal{C}$ be a gluing of the clutters $\mathcal{C}_{1}$ and 
$\mathcal{C}_{2}$ at the vertex $v$. Then 
$\mathrm{height}\,P_{T}(\mathcal{C})=\mathrm{height}\,P_{T_{1}}(\mathcal{C}_{1})+\mathrm{height}\,P_{T_{2}}(\mathcal{C}_{2})$, 
for all $T\in \mathscr{C}(\mathcal{C})$, $T_{i}\in \mathscr{C}(\mathcal{C}_{i})$, 
for $i=1,2$, defined as in Lemma 3.5.
\end{corollary}

\proof 
$\mathcal{C}$ is a gluing of $\mathcal{C}_{1}$ and $\mathcal{C}_{2}$ 
at the free vertex $v$, therefore, $G^{\mathcal{C}}$ is a gluing of 
$G^{\mathcal{C}_{1}}$ and $G^{\mathcal{C}_{2}}$ at $v$. By 
\cite{raufrin} (Corollary 2.4), we have 
\begin{align*} 
\mathrm{height}\,P_{T}(\mathcal{C})= \mathrm{height}\,P_{T}(G^{\mathcal{C}}) &=\mathrm{height}\,P_{T_{1}}(G^{\mathcal{C}_{1}})+\mathrm{height}\,P_{T_{2}}(G^{\mathcal{C}_{2}})\\ &=\mathrm{height}\,P_{T_{1}}(\mathcal{C}_{1})+\mathrm{height}\,P_{T_{2}}(\mathcal{C}_{2}). \qed
\end{align*}

\begin{lemma}\label{glu7}
Let $\mathcal{C}$ be a clutter. The following conditions are equivalent:
\begin{enumerate}
\item[(a)] $J_{\mathcal{C}}$ is unmixed.
\item[(b)] For all $T\in\mathscr{C}(\mathcal{C})$, we have 
$c(T)=\vert T\vert +1$.
\end{enumerate}
\end{lemma}

\begin{proof}
$J_{\mathcal{C}}=J_{G^{\mathcal{C}}}$ and 
$\mathscr{C}(\mathcal{C})=\mathscr{C}(G^{\mathcal{C}})$, 
the proof now follows from \cite{raufrin} (Lemma 2.5).
\end{proof}

\begin{proposition}\label{gluunmixed}
Let $\mathcal{C}$ be a gluing of the clutters $\mathcal{C}_{1}$ 
and $\mathcal{C}_{2}$ at the free vertex $v$. Then 
$J_{\mathcal{C}}$ is unmixed if and only if $J_{\mathcal{C}_{1}}$ and $J_{\mathcal{C}_{2}}$ are unmixed.
\end{proposition}

\begin{proof}
$J_{\mathcal{C}}=J_{G^{\mathcal{C}}}$, 
$J_{\mathcal{C}_{i}}=J_{G^{\mathcal{C}_{i}}}$ and 
the result is true for graphs by \cite{raufrin} (Proposition 2.6), hence 
it is true for clutters as well.
\end{proof}

\begin{theorem}\label{gluCM}
Let $\mathcal{C}$ be a gluing of the clutters $\mathcal{C}_{1}$ and $\mathcal{C}_{2}$ at the free vertex $v$. Then $\mathrm{depth}\,S/J_{\mathcal{C}}=\mathrm{depth}\,S_{1}/J_{\mathcal{C}_{1}}+\mathrm{depth}\,S_{2}/J_{\mathcal{C}_{2}}-2$, where 
$S_{i}=K[\{ x_{j},y_{k} \mid j\in V(\mathcal{C}_{1}), \, k\in V(\mathcal{C}_{2})\}]$. In addition, $J_{\mathcal{C}}$ is Cohen-Macaulay if and only if $J_{\mathcal{C}_{1}}$ and $J_{\mathcal{C}_{2}}$ are Cohen-Macaulay.
\end{theorem}

\begin{proof}
$\mathcal{C}$ is a gluing of $\mathcal{C}_{1}$ and $\mathcal{C}_{2}$ at the 
free vertex $v$, therefore $G^{\mathcal{C}}$ is a gluing of 
$G^{\mathcal{C}_{1}}$ and $G^{\mathcal{C}_{2}}$ at $v$. Then from 
\cite{raufrin} (Theorem 2.7) we have 
$$ \mathrm{depth}\,S/J_{G^{\mathcal{C}}}=\mathrm{depth}\,S_{1}/J_{G^{\mathcal{C}_{1}}}+\mathrm{depth}\,S_{2}/J_{G^{\mathcal{C}_{2}}}-2.$$ 
As $J_{\mathcal{C}}=J_{G^{\mathcal{C}}}$ and $J_{\mathcal{C}_{i}}=J_{G^{\mathcal{C}_{i}}}$, for $i=1,2$, the proof follows.
\end{proof}

\begin{definition}\label{gengluing}{\rm 
Let $\mathcal{C}=\mathcal{C}_{1}\cup\cdots\cup \mathcal{C}_{r}$ be a connected clutter satisfying the following properties for all $i,j,k\in [r]$, which are pairwise different. The following statements are true:
\begin{enumerate}
\item[(1)] $\vert V(\mathcal{C}_{i}\cap V(\mathcal{C}_{j})\vert\leq 1$ and $V(\mathcal{C}_{i})\cap V(\mathcal{C}_{j})\cap V(\mathcal{C}_{k}) = \phi; $

\item[(2)] If $V(\mathcal{C}_{i})\cap V(\mathcal{C}_{j})=\{ v\}$, then $v$ 
is a free vertex in $\Delta (\mathcal{C}_{i})$ and $\Delta (\mathcal{C}_{j})$ 
both.
\end{enumerate}
We say that $\mathcal{C}$ is the gluing of $\mathcal{C}_{1}\ldots,\mathcal{C}_{r}$.
}
\end{definition}
\medskip

In order to characterize Cohen-Macaulay binomial edge ideals in this case, we associate with $\mathcal{C}$ a graph $G^{\mathcal{C}}_{f}$, whose vertex set 
is $V(G^{\mathcal{C}}_{f}) = \{ 1,\ldots,r\}$ and the edge set is 
$E(G^{\mathcal{C}}_{f}) = \{\{ i,j\} : V(\mathcal{C}_{i})\cap V(\mathcal{C}_{j})\not = \phi \}$. The graph $G^{\mathcal{C}}_{f}$ is a connected graph 
since $\mathcal{C}$ is a connected clutter.

\begin{corollary}\label{glu8}
Let $\mathcal{C}=\mathcal{C}_{1}\cup\cdots\cup \mathcal{C}_{r}$ be a connected clutter satisfying properties (1), (2), and assume that the graph $G^{\mathcal{C}}_{f}$ is a tree. Let $S_{i}= K[\{ x_{j}, y_{j} : j\in V(\mathcal{C}_{i})\}]$, for $i=1,\ldots,r$. Then 
$$ \mathrm{depth}\, S/J_{\mathcal{C}}= \mathrm{depth}\, S_{1}/J_{\mathcal{C}_{1}} +\cdots +\mathrm{depth}\, S_{r}/J_{\mathcal{C}_{r}} - 2(r-1).$$
Moreover, $J_{\mathcal{C}}$ is Cohen-Macaulay if and only if each $J_{\mathcal{C}_{i}}$
is Cohen-Macaulay for $i=1,\ldots,r$.
\end{corollary}

\begin{proof}
Since $\mathcal{C}$ is the gluing of $\mathcal{C}_{1},\ldots,\mathcal{C}_{r}$, $G^{\mathcal{C}}$ is the gluing of $G^{\mathcal{C}_{1}},\ldots,G^{\mathcal{C}_{r}}$. 
It is given that $G^{\mathcal{C}}_{f}$ is a tree. Therefore, from \cite{raufrin} (Corollary 2.8) we have 
$$ \mathrm{depth}\, S/J_{G^{\mathcal{C}}}= \mathrm{depth}\, S_{1}/J_{G^{\mathcal{C}_{1}}} +\cdots +\mathrm{depth}\, S_{r}/J_{G^{\mathcal{C}_{r}}} - 2(r-1).$$
Now we have $J_{\mathcal{C}}=J_{G^{\mathcal{C}}}$ and $J_{\mathcal{C}_{i}}=J_{G^{\mathcal{C}_{i}}}$, hence the assertion follows.
\end{proof}

\begin{corollary}\label{glu9}
Let $\mathcal{C}$ be a clutter such that $G^{\mathcal{C}}$ is a chordal graph and 
$\mathcal{C}=\mathcal{C}_{1}\cup\cdots\cup \mathcal{C}_{r}$, such that $\vert V(\mathcal{C}_{i})\cap
V(\mathcal{C}_{j})\vert\leq 1$ for $i\not= j\in\{ 1,\ldots,r\}$. Assume that each $\mathcal{C}_{i}$ is maximal clique. Then the following conditions are equivalent:
\begin{enumerate}
\item[(a)] $J_{\mathcal{C}}$ is Cohen-Macaulay;
\item[(b)] $J_{\mathcal{C}}$ is unmixed;
\item[(c)] $V(\mathcal{C}_{i})\cap V(\mathcal{C}_{j})\cap V(\mathcal{C}_{k})=\phi$ for $i\not=j\not=k\in [r]$.
\end{enumerate}
\end{corollary}

\begin{proof}
$G^{\mathcal{C}}$ is a chordal graph and $G^{\mathcal{C}}=G^{\mathcal{C}_{1}}\cup\cdots\cup G^{\mathcal{C}_{r}}$, such that $\vert V(G^{\mathcal{C}_{i}})\cap
V(G^{\mathcal{C}_{j}})\vert\leq 1$ for $i\not= j\in\{ 1,\ldots,r\}$. Each $G^{\mathcal{C}_{i}}$ is a maximal clique as $\mathcal{C}_{i}$ is maximal clique. Then using \cite{raufrin} (Lemma 2.10 and Corollary 2.11), the proof follows because $J_{\mathcal{C}}=J_{G^{\mathcal{C}}}$.
\end{proof}

\section{Cones on Clutters and Binomial Edge Ideals}
In this section we first define the cone of a clutter and prove the 
unmixed and the Cohen-Macaulay properties of the binomial edge ideal 
of the cone. An important observation is Theorem \ref{coneconverseraufrin}.

\begin{definition}\label{cone}{\rm Let $\mathcal{D}$ be a clutter and $v\not\in V(\mathcal{D})$ be a vertex. Then $\mathcal{C}= \mathrm{cone}\,(v,\mathcal{D})$ is a clutter defined as $V(\mathcal{C})=V(\mathcal{D})\cup\{ v\}$ 
and $E(\mathcal{C})= \{\{ v,i\} \mid i\in V(\mathcal{D})\} \cup E(\mathcal{D})$.
}
\end{definition}

\begin{proposition}\label{cone1}
Let $\mathcal{C}$ be a clutter. If $\mathcal{C}= \mathrm{cone}\,(v,\mathcal{D})$,  
then $G^{\mathcal{C}}= \mathrm{cone}\,(v,G^{\mathcal{D}})$.
\end{proposition}

\begin{proof}
We are given $\mathcal{C}= \mathrm{cone}\,(v,\mathcal{D})$. Then,  $V(G^{\mathcal{C}})=V(\mathcal{C})=V(\mathcal{D})\cup\{ v\} = V(G^{\mathcal{D}})\cup\{ v\} = V(\mathrm{cone}\,(v,G^{\mathcal{D}}))$. Let $\{ i,j\}\in E(G^{\mathcal{C}})$. Then, there exists $e\in E(\mathcal{C})$, such that $\{ i,j\}\subset e$. 
Now $e\in E(\mathcal{C})=E(\mathrm{cone}\,(v,\mathcal{D}))$ implies that 
either $e=\{ v,k\}$, where $k\in V(\mathcal{D})$ or $e\in E(\mathcal{D})$. 
For the first case, $e=\{ v,k\}=\{ i,j\}$. Now $k\in V(\mathcal{D})=V(G^{\mathcal{D}})$ implies that $\{ i,j\}=\{ v,k\}\in E(\mathrm{cone}\,(v,G^{\mathcal{D}}))$. For the second case, $\{ i,j\}\subset e\in E(\mathcal{D})$ implies that 
$\{ i,j\}\in E(G^{\mathcal{D}})$, so $\{ i,j\}\in E(\mathrm{cone}\,(v,G^{\mathcal{D}}))$. Therefore $E(G^{\mathcal{C}})\subseteq E(\mathrm{cone}\,(v,G^{\mathcal{D}}))$.
Again, let $\{ i,j \}\in E(\mathrm{cone}\,(v,G^{\mathcal{D}}))$. Then, 
either $\{ i,j\}=\{ v,k\}$ with $k\in V(G^{\mathcal{D}})$ or 
$\{ i,j\}\in E(G^{\mathcal{D}})$. For $\{ i,j\}=\{ v,k\}\,, k\in V(G^{\mathcal{D}})=V(\mathcal{D})$ we have $\{ i,j\}\in E(\mathrm{cone}\,(v,\mathcal{D}))=E(\mathcal{C})$, so $\{ i,j\} \in E(G^{\mathcal{C}})$. Now if $\{ i,j\}\in E(G^{\mathcal{D}})$, then there exists $e\in E(\mathcal{D})$ such that $\{ i,j\}\subset e$. Since $e\in E(\mathcal{D})$, we also have $e\in E(\mathcal{C})$ which 
implies that $\{ i,j\}\in E(G^{\mathcal{C}})$. Hence $E(G^{\mathcal{C}})=E(\mathrm{cone}\,(v,G^{\mathcal{D}}))$, 
and so $G^{\mathcal{C}}= \mathrm{cone}\,(v,G^{\mathcal{D}})$.
\end{proof}

\begin{example}\label{coneex}{\rm 
Let $\mathcal{D}$ be a clutter with vertex set $V(\mathcal{D})=\{ 1,2,3,4,5,6\}$ and the edge set $E(\mathcal{D})=\{\{ 1,2,4\},\{ 2,4,6\},\{ 4,5\},\{ 1,3,6\}\}$. Consider $\mathcal{C}=\mathrm{cone}\,(7,\mathcal{D})$. Then $E(\mathcal{C})=\{\{ 1,7\},\{ 2,7\},\{ 3,7\},\{ 4,7\},\{ 5,7\},\{ 6,7\},$\\$\{ 1,2,4\},\{ 2,4,6\}, \{ 4,5\},\{ 1,3,6\}\}.$ Let $\mathcal{D}^{\prime}$ be a clutter such that $V(\mathcal{D}^{\prime})=V(\mathcal{D})$ and $E(\mathcal{D}^{\prime})=\{ \{ 1,2\},\{ 2,4\},\{ 1,4\},\{ 2,6\},\{ 4,6\},\{ 4,5\},\{ 1,3\},\{ 3,$\\$6\},\{ 1,6\}\}.$ Then $G^{\mathcal{D}^{\prime}}=\mathcal{D}^{\prime}=G^{\mathcal{D}}$. Therefore we have $G^{\mathcal{C}}=\mathrm{cone}\,(7,G^{\mathcal{D}^{\prime}})$ but $\mathcal{C}\not=\mathrm{cone}\,(7,\mathcal{D}^{\prime})$. Therefore, the converse of the above proposition is not true.
}
\end{example}

\begin{lemma}\label{cone2}
Let $\mathcal{D}$ be a connected clutter, and let $\mathcal{C}= \mathrm{cone}\,(v,\mathcal{D})$. Then, $$ \mathscr{C}(\mathcal{C})=\{ T\subset V(\mathcal{C}) \mid T=T^{\prime}\cup \{ v\}\,\, \mbox{with}\,\, T^{\prime}\not=\phi\,\,\mbox{and}\,\, T^{\prime}\in\mathscr{C}(\mathcal{D})\}\cup\{\phi\}.$$
Moreover, $\mathrm{height}\, P_{T} = \mathrm{height}\, P_{T^{\prime}} + 2$, for all $T\not=\phi$.
\end{lemma}

\proof
It is given that $\mathcal{C}= \mbox{cone}\,(v,\mathcal{D})$, 
therefore $G^{\mathcal{C}}= \mbox{cone}\,(v,G^{\mathcal{D}})$. From \cite{raufrin} (Lemma 3.1), we have
\begin{align*}
\mathscr{C}(\mathcal{C})&=\mathscr{C}(G^{\mathcal{C}})\\ &=\{ T\subset V(G^{\mathcal{C}}) \mid T=T^{\prime}\cup \{ v\}\,\, \mbox{with}\,\, 
T^{\prime}\not=\phi\,\,\mbox{and}\,\, T^{\prime}\in\mathscr{C}(G^{\mathcal{D}})\}\cup\{\phi\} \\ &= \{ T\subset V(\mathcal{C}) \mid T=T^{\prime}\cup \{ v\}\,\, \mbox{with}\,\, T^{\prime}\not=\phi\,\, \mbox{and}\,\, T^{\prime}\in\mathscr{C}(\mathcal{D})\}\cup\{\phi\}.
\end{align*}
Moreover,
\begin{align*}
 \mathrm{height}\, P_{T}(\mathcal{C})=\mbox{height}\, P_{T}(G^{\mathcal{C}})& = \mathrm{height}\, P_{T^{\prime}}(G^{\mathcal{D}}) + 2,\,\,\forall\,\, T\not=\phi\\&= \mathrm{height}\, P_{T^{\prime}}(\mathcal{D}) + 2,\,\,\forall\,\, T\not=\phi. \qed
\end{align*}

\begin{corollary}\label{cone3}
Let $\mathcal{D}$ be a connected clutter and let $\mathcal{C}= \mathrm{cone}\,(v,\mathcal{D})$, with $\vert V(\mathcal{C})\vert=n$. Then $\mathrm{dim}\,S/J_{\mathcal{C}}=\mathrm{max}\{ n+1, \mathrm{dim}\,S^{\prime}/J_{\mathcal{D}}\}$, 
where $S=K[\{ x_{i},y_{i} \mid i\in V(\mathcal{C})\}]$ and 
$S^{\prime}=K[\{ x_{i},y_{i} \mid i\in V(\mathcal{D})\}].$
\end{corollary}

\begin{proof}
Since $\mathcal{C}= \mathrm{cone}\,(v,\mathcal{D})$ defines $G^{\mathcal{C}}= \mathrm{cone}\,(v,G^{\mathcal{D}})$, from \cite{raufrin} (corollary 3.2) we have 
$ \mathrm{dim}\,S/J_{G^{\mathcal{C}}}=\mathrm{max}\{ n+1, \mathrm{dim}\,S^{\prime}/J_{G^{\mathcal{D}}}\}$. Now the proof follows from the 
facts that $J_{G^{\mathcal{C}}}=J_{\mathcal{C}}$ and $J_{G^{\mathcal{D}}}=J_{\mathcal{D}}$.
\end{proof}

\begin{theorem}\label{cone4}
Let $\mathcal{D}$ be a connected clutter and assume that $J_{\mathcal{D}}$ is unmixed. Let $\mathcal{C} =\mathrm{cone}\,(v,\mathcal{D})$. Then the following conditions are equivalent:
\begin{enumerate}
\item[(a)] $\mathcal{D}$ is a complete clutter;
\item[(b)] $J_{\mathcal{C}}$ is unmixed.
\end{enumerate}
If the equivalent conditions hold, then $J_{\mathcal{C}}$ is Cohen-Macaulay.
\end{theorem}

\begin{proof}
(a) $\Longrightarrow$ (b): $\mathcal{D}$ is complete implies that $G^{\mathcal{D}}$ is complete and $\mathcal{C}= \mathrm{cone}\,(v,\mathcal{D})$ implies that  $G^{\mathcal{C}}= \mathrm{cone}\,(v,G^{\mathcal{D}})$. Then, from \cite{raufrin} (Theorem 3.3), we have $J_{G^{\mathcal{C}}}=J_{\mathcal{C}}$ is unmixed.
\medskip

\noindent (b) $\Longrightarrow$ (a): $J_{\mathcal{C}}=J_{G^{\mathcal{C}}}$ 
is unmixed implies that $G^{\mathcal{D}}$ is a complete graph by \cite{raufrin} (Theorem 3.3). Therefore, $\mathcal{D}$ is a complete clutter.
\end{proof}

\begin{lemma}\label{cone5}
Let $\mathcal{D}=\sqcup_{i=1}^{r}\mathcal{D}_{i}$ be a clutter, such that each $\mathcal{D}_{i}$ is connected component with $r\geq 1$, and let $\mathcal{C}= \mathrm{cone}\,(v,\mathcal{D})$. If $J_{\mathcal{C}}$ is unmixed then 
$\mathcal{D}$ has at most two connected components.
\end{lemma}

\begin{proof}
$\mathcal{D}=\sqcup_{i=1}^{r}\mathcal{D}_{i}$ implies that 
$G^{\mathcal{D}}=\sqcup_{i=1}^{r}G^{\mathcal{D}_{i}}$, where 
$G^{\mathcal{D}_{i}}$ is connected. Also, 
$\mathcal{C}= \mathrm{cone}\,(v,\mathcal{D})$ implies that 
$G^{\mathcal{C}}= \mathrm{cone}\,(v,G^{\mathcal{D}})$. Now 
from \cite{raufrin} (Lemma 3.4) we can say that 
$J_{\mathcal{C}}=J_{G^{\mathcal{C}}}$ is unmixed implies that 
$G^{\mathcal{D}}$ has at most two connected components. Therefore, 
$\mathcal{D}$ has at most two connected components.
\end{proof}

\begin{lemma}\label{cone6}
Let $\mathcal{D}=\mathcal{D}_{1}\sqcup \mathcal{D}_{2}$, such that $\mathcal{D}_{1},\,\mathcal{D}_{2}$ are connected clutters, and let $\mathcal{C}= \mathrm{cone}\,(v,\mathcal{D})$. Then, 
$$ \mathscr{C}(\mathcal{C})=\{ T\subset V(\mathcal{C}) : T=T_{1}\cup T_{2}\cup \{ v\},\,T_{i}\in\mathscr{C}(\mathcal{D}_{i})\,\,for\,\, i=1,2\}\cup\{\phi\}.$$
Moreover, $\mathrm{height}\, P_{T} = \mathrm{height}\, P_{T_{1}}+\mathrm{height}\, P_{T_{2}} + 2$, for all $T\not=\phi$.
\end{lemma}

\begin{proof}
$\mathcal{D}=\mathcal{D}_{1}\sqcup \mathcal{D}_{2}$ imply $G^{\mathcal{D}}=G^{\mathcal{D}_{1}}\sqcup G^{\mathcal{D}_{2}}$, where $G^{\mathcal{D}_{1}}, \, G^{\mathcal{D}_{2}}$ connected and $\mathcal{C}= \mathrm{cone}\,(v,\mathcal{D})$ imply $G^{\mathcal{C}}= \mathrm{cone}\,(v,G^{\mathcal{D}})$. Since $\mathscr{C}(\mathcal{C})=\mathscr{C}(G^{\mathcal{C}})$ and $\mathscr{C}(\mathcal{D}_{i})=\mathscr{C}(G^{\mathcal{D}_{i}})$ for $i=1,2$, from \cite{raufrin} (Lemma 3.5) the result follows.
\end{proof}

\begin{corollary}\label{cone7}
Let $\mathcal{D}=\mathcal{D}_{1}\sqcup \mathcal{D}_{2}$ such that $\mathcal{D}_{1},\,\mathcal{D}_{2}$ are connected clutters, and let $\mathcal{C}= \mathrm{cone}\,(v,\mathcal{D})$. 
Then $ \mathrm{dim}\,S/J_{\mathcal{C}}=\mathrm{max}\{ \mathrm{dim}\,S_{1}/J_{\mathcal{D}_{1}} +\mathrm{dim}\,S_{2}/J_{\mathcal{D}_{2}}, n+1\}$, 
where $S=K[\{ x_{k},y_{k} \mid k\in V(\mathcal{C})\}]$ and $S_{i}=K[\{ x_{j},y_{j} \mid j\in V(\mathcal{D}_{i})\}]$, for $i=1,2$. 
\end{corollary}

\begin{proof}
$\mathcal{D}=\mathcal{D}_{1}\sqcup \mathcal{D}_{2}$ implies that 
$G^{\mathcal{D}}=G^{\mathcal{D}_{1}}\sqcup G^{\mathcal{D}_{2}}$ and 
$G^{\mathcal{D}_{1}}, \, G^{\mathcal{D}_{2}}$ connected. Also, 
$\mathcal{C}= \mathrm{cone}\,(v,\mathcal{D})$ implies that 
$G^{\mathcal{C}}= \mathrm{cone}\,(v,G^{\mathcal{D}})$. Since $J_{\mathcal{C}}=J_{G^{\mathcal{C}}}$ and $J_{\mathcal{D}_{i}}=J_{G^{\mathcal{D}_{i}}}$, 
for $i=1,2$, from \cite{raufrin} (Corollary 3.6) the result follows.
\end{proof}

\begin{corollary}\label{cone8}
Let $\mathcal{D}=\mathcal{D}_{1}\sqcup \mathcal{D}_{2}$, such that 
$\mathcal{D}_{1},\,\mathcal{D}_{2}$ are connected clutters. Let 
$\mathcal{C}= \mathrm{cone}\,(v,\mathcal{D})$. The following 
conditions are equivalent:
\begin{enumerate}
\item[(a)] $J_{\mathcal{D}_{1}}$ and $J_{\mathcal{D}_{2}}$ are unmixed;
\item[(b)] $J_{\mathcal{C}}$ is unmixed.
\end{enumerate}
\end{corollary}

\begin{proof} Follows from \cite{raufrin} (Corollary 3.7).
\end{proof}

\begin{theorem}[\cite{raufrin}, Theorem 3.8] \label{coneraufrin}
Let $H=H_{1}\sqcup H_{2}$, such that $H_{1},\,H_{2}$ are connected graphs, 
and let $G= \mathrm{cone}\,(v,H)$. If $J_{H_{1}}$ and $J_{H_{2}}$ are 
Cohen-Macaulay, then $J_{G}$ is Cohen-Macaulay.
\end{theorem}
\medskip

It is a question mentioned in \cite{raufrin}, whether the converse 
of Theorem \ref{coneraufrin}, mentioned above, is true or not. We 
prove in Theorem \ref{coneconverseraufrin} below that the converse 
is indeed true. Let us first recall a result from \cite{kusa}, which 
we would require for proving Theorem \ref{coneconverseraufrin}.

\begin{theorem}[\cite{kusa}, Theorem 3.9]\label{conekusa}
Let $G = \mathrm{cone}\,(v, H)$, where $H$ is a disconnected graph on $[n]$. 
Then 
$\mathrm{depth}\,(S/J_{G}) = \mathrm{min}\,\{ \mathrm{depth}\,(S_{1}/J_{H}), n + 2\}$, where $S_{1}=K[\{ x_{k},y_{k} : k\in V(H)\}]$ and $S=K[\{ x_{k},y_{k} : k\in V(G)\}]$.
\end{theorem}

\begin{theorem}\label{coneconverseraufrin}
Let $H=H_{1}\sqcup H_{2}$, such that $H_{1}$ and $H_{2}$ are connected graphs, 
and let $G= \mathrm{cone}\,(v,H)$. If $J_{G}$ is Cohen-Macaulay then 
$J_{H_{1}}$ and $J_{H_{2}}$ are Cohen-Macaulay.
\end{theorem}

\begin{proof}
Let $\vert V(G)\vert=n$. Since $J_{G}$ is Cohen-Macaulay, $J_{G}$ is unmixed and so $J_{H_{1}}$ and $J_{H_{2}}$ are unmixed by \cite{raufrin} (Corollary 3.7). Therefore $\mathrm{dim}\,(S_{1}/J_{H})=2(n-1)-\mathrm{height}\,J_{H_{1}}-\mathrm{height}\,J_{H_{2}}=n+1$. Now, $\mathrm{depth}\,S/J_{G}=\mathrm{dim}\,S/J_{G}=n+1$, as $J_{G}$ is Cohen-Macaulay. Then, by Theorem \ref{conekusa}, we have 
$\mathrm{depth}\,(S_{1}/J_{H})\geq n+1$. So we have $\mathrm{depth}\,(S_{1}/J_{H})=\mathrm{dim}\,(S_{1}/J_{H})=n+1$, and hence $S_{1}/J_{H}$ is Cohen-Macaulay, equivalently $J_{H_{1}}$ and $J_{H_{2}}$ are Cohen-Macaulay.
\end{proof}

\begin{theorem}\label{cone9}
Let $\mathcal{D}=\mathcal{D}_{1}\sqcup \mathcal{D}_{2}$, such that 
$\mathcal{D}_{1},\,\mathcal{D}_{2}$ are connected clutters, and 
let $\mathcal{C}= \mathrm{cone}\,(v,\mathcal{D})$. Then, 
$J_{\mathcal{D}_{1}}$ and $J_{\mathcal{D}_{2}}$ are Cohen-Macaulay 
if and only if $J_{\mathcal{C}}$ is Cohen-Macaulay.
\end{theorem}

\begin{proof}
We have $J_{\mathcal{D}_{i}}=J_{G^{\mathcal{D}_{i}}}$, for $i=1,2$, $J_{\mathcal{C}}=J_{G^{\mathcal{C}}}$ and $G^{\mathcal{C}}= \mathrm{cone}\,(v,G^{\mathcal{D}})$. 
Therefore, the proof follows from Theorem \ref{coneraufrin} and Theorem \ref{coneconverseraufrin}.
\end{proof}

\bibliographystyle{amsalpha}

\end{document}